\documentclass[graybox]{svmult}

\newcommand{\norm}[1]{\left\lVert#1\right\rVert}
\usepackage{amsmath}
\usepackage{amssymb}
\usepackage{easyReview}

\usepackage{mathptmx}       
\usepackage{helvet}         
\usepackage{courier}        
\usepackage{type1cm}        
\usepackage{makeidx}         
\usepackage{graphicx}        
\usepackage{multicol}        
\usepackage[bottom]{footmisc}

\makeindex             

\begin{document}

\title*{ Distributed model independent algorithm for spacecraft synchronization under relative measurement bias}
\titlerunning{Spacecraft synchronization under relative measurement bias} 
\author{Himani Sinhmar and Sukumar Srikant}
\institute{Himani Sinhmar \at Undergraduate Student, Department of Aerospace Engineering, Indian Institute of Technology Bombay, India. \email{himani.sinhmar@gmail.com}
\and Sukumar Srikant \at Associate Professor, Systems and Control Engineering, Indian Institute of Technology Bombay, India. \email{srikant@sc.iitb.ac.in}}
%
%
\maketitle{}

\abstract*{This paper addresses the problem of distributed coordination control of spacecraft formation. It is assumed that the agents measure relative positions of each other with a non-zero, unknown constant sensor bias. The translational dynamics of the spacecraft is expressed in Euler-Lagrangian form. We propose a novel distributed, model independent control law for synchronization of networked Euler Lagrange system with biased measurements. An adaptive control law is derived based on Lyapunov analysis to estimate the bias. The proposed algorithm ensures that the velocities converge to that of leader exponentially while the positions converge to a bounded neighborhood of the leader positions. We have assumed a connected leader-follower network of spacecraft. Simulation results on a six spacecraft formation corroborate our theoretical findings.}

\abstract{This paper addresses the problem of distributed coordination control of spacecraft formation. It is assumed that the agents measure relative positions of each other with a non-zero, unknown constant sensor bias. The translational dynamics of the spacecraft is expressed in Euler-Lagrangian form. We propose a novel distributed, model independent control law for synchronization of networked Euler Lagrange system with biased measurements. An adaptive control law is derived based on Lyapunov analysis to estimate the bias. The proposed algorithm ensures that the velocities converge to that of leader exponentially while the positions converge to a bounded neighborhood of the leader positions. We have assumed a connected leader-follower network of spacecraft. Simulation results on a six spacecraft formation corroborate our theoretical findings. }

\section{Introduction}
\label{intro}
Spacecraft formation flying is one of the most important technological challenges for modern day space agencies with application to areas like synthetic aperture radars and deep space exploration \cite{JPL_page}. These missions require that spacecraft maintain a desired relative position and attitude at all times. In  synchronization problems \textit{consensus} is the significant objective and implies that all the agents reach an agreement on a common value by locally interacting with their neighbors. In distributed multi-agent coordination problems (distributed algorithm allows the agents to execute control law without requiring information of the network as a whole), point models are generally considered due to their simplicity but are not realistic. Euler--Lagrange equations can be used to model a large class of aero-mechanical systems including autonomous vehicles and spacecraft in formation \cite{weiren}. Networked Lagrangian systems are studied in detail in \cite{weiren}, where the authors propose consensus algorithms accounting for actuator saturation and for unavailability of measurements of generalized coordinates. In \cite{model_el} formation dynamics of spacecraft formation is discussed, describing the dynamics in Euler-Lagrangian form.

Distributed and model independent algorithms for directed networks in the presence of bounded disturbance is addressed in \cite{anderson}. In \cite{china} a control law to achieve finite time coordinated control for 6DOF spacecraft formation is developed. However, this algorithm is model dependent and requires the knowledge of self states of the agents, and further the gravitational and centrifugal forces acting on them.  A model dependent control law is designed in \cite{slotine} using contraction analysis for synchronization of spacecraft. In \cite{krogstad}, a synchronization controller for attitude and position control of a spacecraft formation is designed which rely on all to all communication topology. An algorithm for tracking of Lagrangian systems using only position measurements is developed in \cite{only_pos} by encompassing a distributed observer to estimate unknown velocity of the agents. An output feedback structured model reference adaptive control (MRAC) law has been developed for spacecraft rendezvous in \cite{singla}. However, their control law works well only in the presence of bounded disturbances and measurement errors. In \cite{ext_dist}, the coordination control problem of heterogeneous first and second order multi-agent systems with external disturbances is considered, but the disturbances are assumed to be $\mathbb{L}^2$ bounded. In \cite{nl_wang}, a composite consensus control strategy is proposed for second-order multi-agent systems with mismatched bounded disturbances. 

In the aforementioned literature on consensus with errors, adaptive control algorithms in the presence of an upper bound on disturbances and stochastic errors have been studied. But what happens to consensus in the presence of measurement errors with unknown bounds? The current work addresses this problem. Further, strategies for handling disturbance do not usually fare well for the case of measurement errors simply due to the fact that the measurement errors scale with the control gain while disturbances external to the system do not. This makes ensuring bounded trajectories with constant measurement bias a much harder problem than the disturbance robustness case.
The relevant contributions in the domain of measurement bias errors known to the authors are by \cite{tandon} and \cite{mrac}. While the former proposes an adaptive control law in the presence of constant bias for a double integrator system, the latter addresses the problem of accommodating unknown sensor bias in a direct MRAC setting for state tracking using state feedback. Motivated by the above work, we present a distributed model independent synchronization algorithm for a spacecraft network described in Euler-Lagrangian form and achieving consensus to a neighborhood in the presence of an unknown, unbounded and constant sensor bias in the measurement of relative position. 

\section{Preliminaries}
\label{prelim}
In this section we present several notations, lemmas, assumptions and an introduction on graph theory for subsequent use. 

\subsection{Mathematical Notations}
Given a vector \textbf{x} = $[x_1,...,x_n]^T$ $\in$ $\mathbb{R}^n$, $sgn(\mathbf{x})$ = $[sgn(x_1),...,sgn(x_n)]^T$, where $sgn(\cdot)$ is the standard signum function, $\mathbf{1}_n = [1,..,1]^T$ and $\mathbf{0}_n = [0,...,0]^T$. 
One-norm and Euclidean norm of a vector $\mathbf{x}$ are denoted by $\norm{\textbf{x}}_1$ = $\sum_{i=1}^{n} |x|_i$ and $\norm{\textbf{x}}$ = $(\sum_{i=1}^{n} |x|_i^2)^{\frac{1}{2}}$ respectively. A Diagonal matrix with diagonal elements as $d_1,d_2,...,d_n$ is represented by $diag(d_1,...,d_n)$ and a block diagonal matrix with diagonal matrices $B_1,...,B_n$ is represented by $blkdiag(B_1,...,B_n)$. A $n\times n$ identity matrix is denoted by $\vec{I}_{n \times n}$ and a matrix with all elements as zero is denoted by $\vec{0}_{n \times n}$. We use $\otimes$ to denote Kronecker product.

\subsection{Graph Theory}

Consider a multi-agent system with $n$ agents interacting with each other through a communication or sensing network or a combination of both. This network is modeled as either \textit{undirected} or \textit{directed} graph. We define the graph, $\mathcal{G} \triangleq (\mathcal{V},\mathcal{E})$, where $\mathcal{V} \triangleq {1,...,n}$ is a node set and $\mathcal{E} \subseteq \mathcal{V} \times \mathcal{V} $ is an edge set of nodes, called edges \cite{wr_prelim}. An edge $(i,j)$ in the edge set of a directed graph signifies that agent $j$ can obtain information from agent $i$ but not vice-versa. If an edge $(i,j) \in \mathcal{E}$, then node $i$ is a neighbor of node $j$. The set of neighbors of node $i$ is denoted by $\mathcal{N}_i$. In an undirected graph the pair of nodes are unordered, where the edge $(i,j)$ denotes that agents $i$ and $j$ can obtain information from each other, i.e. $(j,i) \in \mathcal{E} \Leftrightarrow (i,j) \in \mathcal{E}$. A weighted graph associates a weight with every edge in the graph. An undirected graph is connected if there is an undirected path between every pair of distinct nodes \cite{wr_prelim}. The \textit{adjacency} matrix, $\mathcal{A} = [a_{ij}] \in \mathbb{R}^{n \times n}$, is defined such that $a_{ij}$ is a positive weight if $(j,i) \in \mathcal{E}$ and $a_{ij} = 0$ if $(j,i) \neq \mathcal{E}$. Since no self edges are present, $a_{ii} = 0$. For an undirected graph, $\mathcal{A}$ is symmetric. The \textit{degree} matrix of the graph $\mathcal{G}$ is, $\mathcal{D} = diag(\sum \limits_{j=1}^{n} a_{1j},...,\sum \limits_{j=1}^{n} a_{nj}) \in \mathbb{R}^{n \times n}$. \textit{Laplacian} matrix, $\mathcal{L} \triangleq [l_{ij}] \in \mathbb{R}^{n \times n}$, is then defined as 
\begin{align}
    \mathcal{L} &= \mathcal{D} - \mathcal{A} \nonumber\\
    l_{ii} &= \sum\limits_{j=1,j \neq i}^n a_{ij}, \quad \quad l_{ij} = -a_{ij}, i\neq j
\end{align}
$\mathcal{L}$ is symmetric for undirected graphs and since $\mathcal{L}$ has zero row sums, 0 is an eigenvalue of $\mathcal{L}$ with an associated eigenvector $\vec{1}_n$. Laplacian matrix is diagonally dominant and has non negative diagonal entries \cite{wr_prelim}. Note that, $\mathcal{L}\vec{x}$ is a column stack vector of $\sum \limits_{j=1}^n a_{ij}(x_i - x_j)$, where $\vec{x} = [x_1,...,x_n]^T \in \mathbb{R}^n$. 

For a leader follower network, we let the leader be denoted by 0 and followers by nodes $1,...,n$. The Laplacian matrix of the followers is denoted by $\mathcal{L}$. The communication between the leader and a follower is unidirectional with the leader issuing the communication. The edge weight between the leader follower is denoted by $a_{i0}, i \in \mathcal{V}$. If the $i^{th}$ follower is connected to the leader then $a_{i0} > 0$ and 0 otherwise. We define $\bar{A} = diag(a_{10},...a_{n0})$.

\subsection{Lemmas} 

 \begin{Assumption}
 \label{ass1}
 All followers are connected to the leader and the communication network is undirected.
 \end{Assumption}
 \begin{Assumption}
 \label{ass2}
 Neighbors can exchange both, their measurement of relative position of the leader and their estimate of the bias.
 \end{Assumption}
 \begin{lemma}
 \cite{wr_prelim}\label{pdf_lf}
 If Assumption \ref{ass1} holds then $\mathcal{L} + \bar{A}$ is positive definite. 
 \end{lemma}
 
 \begin{lemma}\cite{wr_prelim}
 If the symmetric matrix $H > 0$ $\forall$ \vec{x} $\in \mathbb{R}^n$, then 
 \begin{align}
     \lambda_{min}(H) \norm{x}^2 &\leq \vec{x}^T H \vec{x} \leq \lambda_{max}(H) \norm{x}^2
 \end{align}
 \end{lemma}
 
\begin{lemma}\cite{wr_prelim}
If graph $\mathcal{G}$ is undirected and  connected, then $\mathcal{L}$ has following properties: 
\begin{enumerate}
    \item For any $\vec{x}\in \mathbb{R}^n$, $\vec{x}^T\mathcal{L}\vec{x} = \frac{1}{2}\sum \limits_{i=1}^n \sum \limits_{j=1}^n a_{ij}(x_i-x_j)^2$ which implies that $\mathcal{L}$ is positive semidefinite
    \item  $\mathcal{L}\vec{x} = 0$ or $\vec{x}^T\mathcal{L}\vec{x} = 0$ if and only if $x_i = x_j$ for all $i,j = 1,...,n$
    \item Let $\lambda_i(\mathcal{L})$ be the $i^{th}$ eigenvalue of $\mathcal{L}$ with $\lambda_1(\mathcal{L})\leq \lambda_2(\mathcal{L}) \leq \cdot \cdot \cdot \leq \lambda_n(\mathcal{L})$, so that $\lambda_1(\mathcal{L}) = 0$. Then, $\lambda_2(\mathcal{L})$ is the \textit{algeberic connectivity}, which is positive if and only if the undirected graph is connected. The algebraic connectivity quantifies the convergence rate of consensus algorithms
\end{enumerate}
\end{lemma}

\begin{lemma}[Barbalat's Lemma] \cite{rac_book}
\label{bl}
If, for a vector-valued function, $f(\cdot):\mathbb{R} \to \mathbb{R}^n$ the following conditions hold true,
  \begin{enumerate}
    \item $lim_{t\to\infty}\int_0^tf(\tau)d\tau$ exists and is finite
    \item $f(t)$ is uniformly continuous
  \end{enumerate}
then, $lim_{t \to \infty} f(t) = 0$. 
\begin{corollary}
If, for a vector-valued function, $f(\cdot):[0, \infty)\to \mathbb{R}^n$ the following two conditions hold true,
  \begin{enumerate}
    \item $f(x) \in \mathbb{L}^\infty \cap \mathbb{L}^p$ for any $p \in [1, \infty)$ and,
    \item $f'(x) \in \mathbb{L}^\infty$
  \end{enumerate}
then, $lim_{x \to \infty} f(x) = 0$
\end{corollary}
\end{lemma}

\begin{lemma}\cite{weiren}
Let \vec{x} and \vec{y} be any two vectors in $\mathbb{R}^n$, $\vec{A} \in \mathbb{R}^{n\times n}$ be a matrix. Then, 
\begin{align}
\vec{x}^T sgn(\vec{x}) &= \norm{\vec{x}}_1 \\
\norm{\vec{x}}_1 &\geq \norm{\vec{x}} \\
|\vec{x}^T \vec{A} \vec{y}| &\leq \norm{\vec{x}}\norm{\vec{A}}\norm{\vec{y}}
\end{align}
\end{lemma}


\subsection{Spacecraft Relative Orbital Dynamics}

For a leader follower spacecraft formation, relative translational orbital dynamics equations are described in \cite{model_el}. The leader orbit frame has its origin located in the centre of mass of the leader spacecraft. 
\begin{figure}
\label{frame}
\centering
\includegraphics[scale = .8]{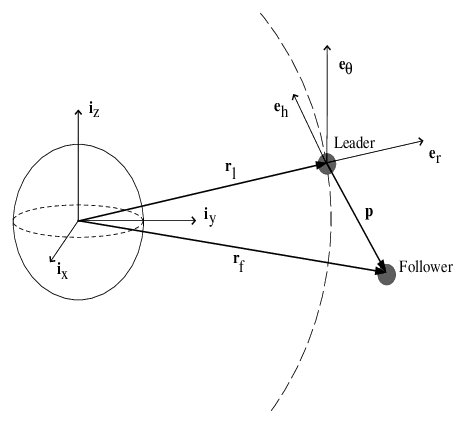}
\caption{Earth centered initial frame ($\vec{i}_{\vec{X}}, \vec{i}_{\vec{y}},\vec{i}_{\vec{Z}}$) and Leader orbit reference frame ($\vec{e}_{\vec{r}}, \vec{e}_{\theta},\vec{e}_{\vec{h}}$)\cite{model_el} }
\label{pb}       
\end{figure}
The $\vec{e}_r$ axis is parallel to $\vec{r}_l$ (vector joining the center of the earth and the leader) and  $\vec{e}_h$ axis is parallel to the orbit momentum vector which points in the orbit normal direction. The  $\vec{e}_{\theta}$ axis completes the right handed orthogonal frame. Non-linear relative motion dynamics for spacecraft in formation is given by \eqref{SFF_eq} : 

\begin{subequations}
\label{SFF_eq}
\begin{align}
    \ddot{x} -2\dot{\theta} \dot{y} + \bigg(\frac{\mu}{r_f^3}-\dot{\theta}^2\bigg)x-\ddot{\theta}y + \mu\bigg(\frac{r_l}{r_f^3}-\frac{1}{r_l^2}\bigg) &= \frac{\tau_x}{m_f} \\
    \ddot{y} + 2\dot{\theta}\dot{x}+\ddot{\theta}x+\bigg(\frac{\mu}{r_f^3}-\dot{\theta}^2\bigg)y &= \frac{\tau_y}{m_f}  \\
    \ddot{z}+\frac{\mu}{r_f^3}z &= \frac{\tau_z}{m_f}
\end{align}
\end{subequations}
where $\vec{r_f}$ is the orbit radius of the follower and $\dot{\theta}$ is the true anomaly rate of the of the leader. $\vec{\tau}$ is the  actuator force of the follower. $\vec{p} = \begin{bmatrix}x&y&z \end{bmatrix}^T$ is the relative position between the leader and follower in leader orbit reference frame. $m_f$ and $m_l$ are the masses of the follower and leader respectively  and $\mu = G(m_l + m_f)$.
\eqref{SFF_eq} can be written in the Euler Lagrangian form for the $i^{th}$ follower as,
\begin{align}
    \label{EL_fol}
    \vec{M}_i\ddot{\vec{q}}_i + \vec{C}_i(\dot{\theta})\dot{\vec{q}}_i + \vec{g}_i(\dot{\theta},\ddot{\theta},\vec{q}_i) &= \boldsymbol{\tau}_i
\end{align}
where
\begin{align}
\label{mcg}
    \vec{M}_i &= \begin{bmatrix}
    m_i & 0 & 0\\ 0 & m_i & 0 \\ 0 & 0 & m_i
    \end{bmatrix} \\
    \vec{C}_i &= \begin{bmatrix}
    0 & -2m_i\dot{\theta} & 0\\
    -2m_i\dot{\theta} & 0 & 0\\
    0 & 0 & 0
    \end{bmatrix} \\
    \vec{g}_i &= m_i\begin{bmatrix}
    \bigg(\frac{\mu}{r_f^3}-\dot{\theta}^2\bigg)x_i -\ddot{\theta}y_i + \mu\bigg(\frac{r_l}{r_f^3}-\frac{1}{r_l^2}\bigg)\\
    \ddot{\theta}x_i + \bigg(\frac{\mu}{r_f^3}-\dot{\theta}^2\bigg)y_i \\
    \frac{\mu z_i}{r_f^3}
    \end{bmatrix}
\end{align}
Here, $\vec{q}_i = \begin{bmatrix} x_i&y_i&z_i \end{bmatrix}^T$ and $\dot{\vec{q}}$ is the relative position and relative translational velocity of the $i^{th}$ agent with respect to the leader in leader orbit reference frame. 
Define $\vec{q} \triangleq [q_1,...,q_n]^T$, $\vec{\dot{q}} \triangleq [\dot{\vec{q}}_1,...,\dot{\vec{q}}_n]^T$, $\boldsymbol{\tau} \triangleq [\boldsymbol{\tau}_1,...,\boldsymbol{\tau}_n]^T$, $\vec{M} \triangleq diag(\vec{M}_1,...,\vec{M}_n)$, $\vec{C} \triangleq diag(\vec{C}_1,..,\vec{C}_n)$ and $\vec{g} \triangleq [\vec{g}_1,...,\vec{g}_n]^T$

\section{Problem Formulation}
We are interested in formation flight of spacecraft described by the following Euler-Lagrange equation,
 \begin{align}
     \vec{M}_i(\vec{q}_i)\ddot{\vec{q}}_i + \vec{C}_i(\vec{q}_i,\dot{\vec{q}}_i)\dot{\vec{q}}_i + \vec{g}_i(\vec{q}_i) = \vec{\tau} _i, \quad \quad i=1,...,n
     \label{EL_gen}
 \end{align}
 where $\vec{q}_i \in \mathbb{R}^p$ is the relative position vector of the $i^{th}$ agent with respect to the leader, $\vec{M}_i(\vec{q}_i) \in \mathbb{R}^{p \times p}$ is the symmetric positive definite inertia matrix, $\vec{C}_i(\vec{q}_i,\dot{\vec{q}}_i)\dot{\vec{q}}_i$ is the vector of Coriolis and centrifugal torques, $\vec{g}_i(\vec{q}_i)$ is the vector of gravitational torques and $\vec{\tau}_i \in \mathbb{R}^p$ is the force produced by the actuator of the $i^{th}$ agent. Here, the leader specifies the objective for the follower network. The agents can measure relative positions using line of sight vector technique and a constant \emph{unknown} bias, $\vec{b}_i \in \mathbb{R}^3$ for $i^{th}$ agent, is present in these measurements. Now, we make the following assumptions:
 \begin{Assumption}
 \label{ass3}
There exist positive constants $k_{m_i}, k_{c_i} $ and $k_g$ such that $M_i(q_i) - k_{m_i} \mathbf{I}_p \leq 0$, $\norm{(g_i(q_i))} \leq k_g$ and $\norm{(C_i(q_i,\dot{q}_i))} \leq k_{c_i}$
 \end{Assumption}
  \begin{Assumption}
  \label{a}
$\dot{M}_i(q_i) - 2C_i(q_i,\dot{q}_i)$ is skew symmetric
 \end{Assumption}
The objective is for the followers to approach the generalized coordinates of the leader with local interaction. We propose a non linear, distributed and model independent adaptive control law which ensures asymptotic convergence to a neighborhood of the consensus. A Lyapunov based analysis is used to derive bias estimator dynamics. 

\section{Control Law Design}

Define
\begin{align}
\label{s}
    s_i &= \dot{\vec{q}}_i + \lambda(\vec{q}_i + \vec{b_i} - \hat{\vec{b}})_i, \quad \lambda \geq 0
\end{align}
where $\vec{b}_i$ is the bias and $\hat{\vec{b}}_i$ is the estimate of the bias for the $i^{th}$ agent. \eqref{EL_fol} can then be written as:
\begin{align} 
\label{el_s}
    \vec{M}_i\dot{s}_i &= \boldsymbol{\tau}_i - \vec{C}_i\dot{\vec{q}}_i - \vec{g}_i + \lambda \vec{M}_i(\dot{\vec{q}}_i - \dot{\hat{\vec{b}}}_i)
\end{align}
We propose the following control law :
\begin{align}
\label{tau1}
    \boldsymbol{\tau}_i &= -\alpha \sum \limits_{j=0}^n a_{ij}(s_i-s_j)-\beta_i sgn(s_i) -\gamma_i\norm{\dot{\vec{q}}_i}sgn(s_i), \quad \alpha, \beta_i, \gamma_i \geq 0 \\
    \label{tau2}
    \boldsymbol{\tau} &= -\alpha[(\mathcal{L} + \bar{A})\otimes \vec{I}_3]s  -\beta sgn(s) - \Gamma Q sgn(s) 
\end{align}
where $\Gamma \triangleq blkdiag(\gamma_1\vec{I}_3,...,\gamma_n\vec{I}_3)$ and $Q \triangleq blkdiag(\norm{\dot{\vec{q}}}_1\vec{I}_3,..., \norm{\dot{\vec{q}}}_n\vec{I}_3)$ and $s \triangleq [s_1,...,s_n]^T$. Define the following placeholders for brevity:
\begin{align}
    H &= \mathcal{L} + \bar{A}\\
    H_1 &= H \otimes \vec{I}_3 \\
    \Tilde{\vec{b}} &= \vec{b} - \Hat{\vec{b}} \\
    \Tilde{\vec{q}} &= \vec{q} + \Tilde{\vec{b}}
\end{align}
The adaptive control law for estimating bias is taken to be:
\begin{align}
\label{adaption}
    \dot{\hat{\vec{b}}} &= -\dot{\vec{q}}
\end{align}

\begin{theorem}
Consider the multi-agent leader follower spacecraft network with agent dynamics given by \eqref{EL_gen} and an undirected connected communication graph $\mathcal{G}$. If Assumptions 1 - 4 hold, then the control law described by \eqref{s}--\eqref{tau2} and bias adaptation law \eqref{adaption}, guarantees that $\lim \limits_{t \to \infty}\dot{\vec{q}}(t) \rightarrow 0$, $\lim \limits_{t \to \infty}[\vec{q}(t) + \tilde{\vec{b}}(t)]\rightarrow 0$ exponentially.
\end{theorem}

\begin{proof}
\smartqed
Consider the following Lyapunov function candidate 
\begin{align}
\label{V}
    V = \frac{1}{2} s^T \vec{M} s
\end{align}
Taking derivative along dynamics and control from \eqref{EL_gen}--\eqref{tau2},
\begin{align}
   \dot{V} &= \frac{1}{2}s^T \dot{\vec{M}} s + s^T \vec{M} \dot{s}\nonumber \\
   &= s^T(-\alpha H_1 s -\beta sgn(s) - \Gamma Q sgn(s) - \vec{C}\dot{\vec{q}} - \vec{g} + \lambda \vec{M}(\dot{\vec{q}}-\dot{\Hat{\vec{b}}})) \nonumber \\
   &= -\alpha s^T H_1 s - \beta \norm{s}_1 -s^T\Gamma Q sgn(s) - s^T \vec{C}\dot{\vec{q}} -s^T\vec{g} + \lambda s^T \vec{M}(\dot{\vec{q}}-\dot{\hat{\vec{b}}})
\end{align}
Further, substituting \eqref{adaption} and using Lemmas 2 and 5, we have
\begin{align}
    \dot{V} &\leq -\alpha s^T H_1 s - \beta \norm{s} + k_g\norm{s}-s^T\Gamma Q sgn(s)-s^T \vec{C}\dot{\vec{q}} + 2\lambda s^T \vec{M} \dot{\vec{q}} \nonumber \\
    &\leq -\alpha s^T H_1 s - (\beta -k_g) \norm{s} -\sum \limits_{i=1}^{n} \gamma_i\norm{\vec{\dot{q}}_i}\norm{s_i}_1 - \sum
    \limits_{i=1}^n s_i^T\vec{C}_i\dot{\vec{q}}_i + 2\lambda \sum \limits_{i=1}^n s^T\vec{M}_i \dot{\vec{q}}_i   \nonumber \\
    &\leq -\alpha s^T H_1 s - (\beta -k_g) \norm{s} -\sum \limits_{i=1}^{n} \gamma_i\norm{\vec{\dot{q}}_i}\norm{s_i} +\sum
    \limits_{i=1}^n \norm{s_i}\norm{\vec{C}_i}\dot{\vec{q}}_i + 2\lambda\sum 
    \limits_{i=1}^n \norm{s_i}\norm{\vec{M}_i}\norm{\dot{\vec{q}}_i}\nonumber \\
    &\leq -\alpha s^T H_1 s - (\beta -k_g) \norm{s} +\sum \limits_{i=1}^{n} (k_{c_i} + 2\lambda k_{m_i}-\gamma_i)\norm{s_i}\norm{\dot{\vec{q}}_i} 
\end{align}
If we choose 
\begin{align}
\label{constants}
    \beta &> k_g \\
    \gamma_i &> k_{c_i} + 2\lambda k_{m_i}
\end{align}
We have
\begin{align}
\label{vdot}
    \dot{V} &\leq -\alpha s^T H_1 s \nonumber \\
    &\leq -\alpha \lambda_{min}(H)\norm{s}^2 
\end{align}
From \eqref{V} we have, $$ V \leq \frac{k_m}{2}\norm{s}^2 \implies \norm{s}^2 \geq \frac{2}{k_m}V$$ Substituting this in \eqref{vdot},
\begin{align}
\label{V_t}
    \dot{V} &\leq -\eta V, \quad \eta = \frac{2\alpha\lambda_{min}(H)}{k_m} \geq 0 \nonumber \\
    V(t) &\leq V(0)e^{-\eta t}
\end{align}
\eqref{V_t} implies $\lim \limits_{t \to 
\infty}V(t) \leq 0$. However, from \eqref{V} we have $V(t) \geq 0$ implying $\lim \limits_{t \to 
\infty}V(t) = 0 \implies \lim\limits_{t \to \infty}s(t) = 0$. Let the initial condition for position, velocity and bias be given by $\vec{q}(0), \dot{\vec{q}}(0)$ and $\tilde{\vec{b}}(0)$ respectively. Using \eqref{adaption} and the fact that $\lim\limits_{t \to \infty}s(t) = 0$ we have,
\begin{align}
\label{s_0}
    \dot{\vec{q}} + \lambda(\vec{q} + \Tilde{\vec{b}}) &= 0
\end{align}
Solving \eqref{adaption} and \eqref{s_0} using Laplace transform we get
\begin{align}
\label{qdot}
    \vec{\dot{q}} &= -\lambda(\vec{q}(0)+\tilde{\vec{b}}(0))e^{-2\lambda t}\\
    \label{q}
    \vec{q} &= \bigg(\frac{\vec{q}(0)+\tilde{\vec{b}}(0)}{2}\bigg)e^{-2\lambda t} + \bigg(\frac{\vec{q}(0)-\tilde{b}(0)}{2}\bigg) \\
    \label{b}
    \tilde{\vec{b}} &= \bigg(\frac{\vec{q}(0)+\tilde{\vec{b}}(0)}{2}\bigg)e^{-2\lambda t} - \bigg(\frac{\vec{q}(0)-\tilde{\vec{b}}(0)}{2}\bigg)
\end{align}
Applying limit on \eqref{qdot}, \eqref{q} and \eqref{b} to analyze the asymptotic behavior :
\begin{align}
\label{limits}
    \lim \limits_{t \to \infty} \dot{\vec{q}}(t) &= 0 \\
    \lim \limits_{t \to \infty} [\vec{q}(t) &+ \tilde{\vec{b}}(t)] = 0 \\
    \lim \limits_{t \to \infty} \vec{q}(t) &= \frac{\vec{q}(0)-\tilde{\vec{b}}(0)}{2}  \\
    \lim \limits_{t \to \infty} \tilde{\vec{b}}(t) &= -\bigg(\frac{\vec{q}(0)-\tilde{\vec{b}}(0)}{2}\bigg)
\end{align}
Hence, using the proposed control law we are able to achieve exponential convergence of velocity ($\dot{\vec{q}}$) and $(\vec{q}+\tilde{\vec{b}})$ as seen from \eqref{qdot}-\eqref{b} while position and bias converges to a constant value in the neighborhood of consensus. 
\qed
\end{proof}

\section{Simulations}
In this section, simulations results are presented to validate our algorithm. We have considered one leader and five agents. All the agents are assumed to be connected to leader. The Laplacian and Adjacency matrix of the followers are given by:
\begin{align}
    \mathcal{L} &= \begin{bmatrix}
    1 & -1 & 0 & 0 & 0\\
    -1 & 2 & -1 & 0 & 0\\
    0 & -1 & 2 & -1 & 0\\
    0 &  0 & -1 & 2 & -1\\
    0 & 0 & 0 & -1 & 1
    \end{bmatrix}, \quad \quad
    \mathcal{A} = \begin{bmatrix}
    0 & 1 & 0 & 0 & 0\\
    1 & 0 & 1 & 0 & 0\\
    0 & 1 & 0 & 1 & 0\\
    0 &  0 & 1 & 0 & 1\\
    0 & 0 & 0 & 1 & 0
    \end{bmatrix}
\end{align}
The reference orbit (leader's orbit) is assumed to be circular with $r_l = 7078 km$ and mass of each follower is identical, $m_i = 1 kg$. The initial relative position of the followers is randomly chosen to lie between $[0,9]m$ while the initial relative velocities lie in $[0,6]m/s$. The true bias lies in the range of $[-1,2]m$ for each coordinate of agents. The initial estimate of bias for $i^{th}$ agent is initialized as, $\hat{\vec{b}_i} = (\vec{b}_i-1) m$. The constants $\alpha$, $\lambda$, $\beta_i$ and $\gamma_i$ are chosen to be 1, 0.5, 20.2 and 3.12 respectively. 

Fig. \ref{pb} shows time variation of the compensated biased relative positions of the followers. It is evident by this figure that $\vec{q} + \vec{\tilde{b}} \rightarrow 0$ exponentially for all the agents. From Fig. \ref{pos} we can observe that the bounded trajectories are obtained for all the followers in the neighborhood of leader's orbit asymptotically. This bound on the trajectory depends on the initial value of the biased position. Fig. \ref{vel} shows the relative velocities of the followers with respect to leader. It can be seen that velocity for all the followers approach to that of leader exponentially. 
\begin{figure}
\includegraphics[scale=0.52]{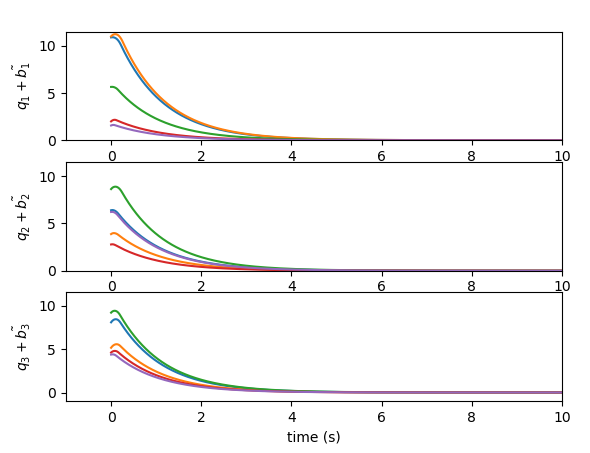}
\caption{$[\vec{q}(t) + \tilde{\vec{b}}(t)](m)$ vs time ($s$). The sum of position and $\tilde{\vec{b}}$ exponentially converges to leader's trajectory}
\label{pb}       
\end{figure}

\begin{figure}
\includegraphics[scale=0.52]{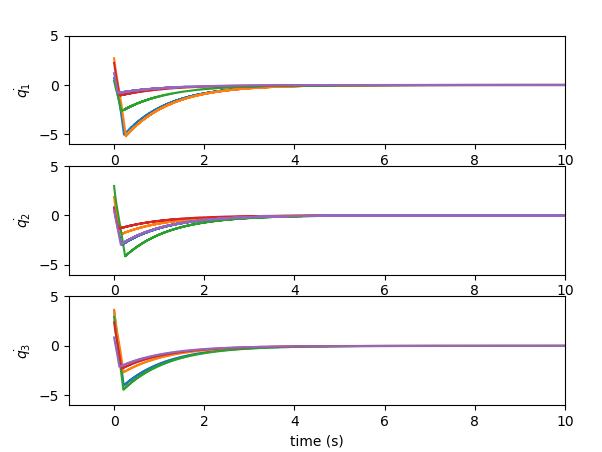}
\caption{$\dot{\vec{q}}(t)(m/s)$ vs time ($s$). The velocity of all agents converges exponentially to the leader's velocity}
\label{vel}       
\end{figure}

\begin{figure}
\includegraphics[scale=.52]{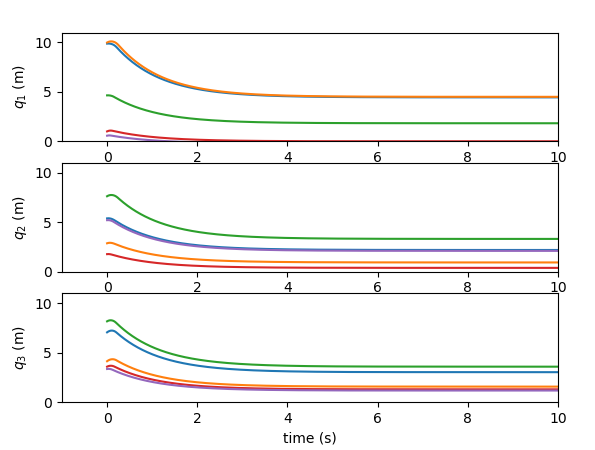}
\caption{$\vec{q}(t) $ vs time. Position of the followers converges to a constant value in the neighborhood of the leader's trajectory}
\label{pos}       
\end{figure}

\section{Conclusion}
A distributed model independent algorithm is proposed in this paper for an undirected connected network governed by Euler-Lagrange dynamics with biased measurements to achieve consensus. It is shown that the velocity and  biased position with bias compensation exponentially converges to the leader's trajectory. No knowledge of upper bounds on the measurement errors are assumed in this work. To ensure stability, control gain matrices are introduced which require the knowledge of upper bounds on the inertia matrix and centrifugal matrix of the system. These requirements are reasonable as we usually know the nominal dynamics of the agents which can directly give the possible bound on these quantities. This algorithm can be easily extended to any Euler-Lagrange system. Simulations are provided to show the effectiveness of this algorithm. In future work, modification in the proposed algorithm will be investigated to reduce the bound on consensus errors. Moreover, actuator saturation also needs to be taken into account.

\end{document}